\documentclass[11pt]{article}
\usepackage{amsmath,amssymb}
\usepackage{graphicx}
\usepackage{tikz}
\usetikzlibrary{arrows.meta, angles, quotes, calc}
\usetikzlibrary{positioning}
\usepackage{amsthm}
\usepackage{comment}
\usepackage{tikz-3dplot}
\usepackage{ulem}

\theoremstyle{definition}
\newtheorem{theorem}{Theorem}
\newtheorem*{theorem*}{Theorem}

\newtheorem{definition}[theorem]{Definition}
\newtheorem*{definition*}{Definition}

\newtheorem{prop}[theorem]{Proposition}
\newtheorem{lemma}[theorem]{Lemma}

\newtheorem{example}[theorem]{Example}
\newtheorem{rem}[theorem]{Remark}
\newtheorem{notation}[theorem]{Notation}
\newtheorem{question}[theorem]{Question}

\numberwithin{theorem}{section}
\numberwithin{equation}{section}

\newcommand{\graphDb}[0]{
\begin{tikzpicture}[x=1pt,y=1pt,yscale=-0.5,xscale=0.5,baseline=-140pt, line width=1pt]

\draw [dash pattern={on 3pt off 2pt}] [-stealth]  (90,270)..controls (130,220)..(165,265) ;
\draw [dash pattern={on 3pt off 3pt}]  [-stealth] (250 ,270).. controls (290, 220).. (325,265) ;

\draw [-stealth] (90, 270)--(165,270) ;
\draw [-stealth](170, 270)--(245,270) ;

\draw  [fill={rgb, 255:red, 0; green, 0; blue, 0 }  ,fill opacity=1 ] (90,270) circle (5);
\draw (90, 290) node {$1$};
\draw  [fill={rgb, 255:red, 0; green, 0; blue, 0 }  ,fill opacity=1 ] (170,270) circle (5);
\draw (170, 290) node {$2$};
\draw [fill={rgb, 255:red, 0; green, 0; blue, 0 }  ,fill opacity=1 ] (250,270) circle (5);
\draw (250, 290) node {$3$};
\draw [fill={rgb, 255:red, 0; green, 0; blue,0 } ,fill opacity=1] (330,270) circle (5);
\draw (330, 290) node {$4$};

\end{tikzpicture}}

\newcommand{\graphD}[0]{
\begin{tikzpicture}[x=1pt,y=1pt,yscale=-0.5,xscale=0.5,baseline=-140pt, line width=1pt]

\draw [dash pattern={on 3pt off 2pt}] [-stealth]  (90,270)..controls (170,220)..(245,265) ;
\draw [dash pattern={on 3pt off 3pt}]  [-stealth] (170,270).. controls (250, 220).. (325,265) ;

\draw [-stealth] (90, 270)--(165,270) ;
\draw [-stealth](170, 270)--(245,270) ;

\draw  [fill={rgb, 255:red, 0; green, 0; blue, 0 }  ,fill opacity=1 ] (90,270) circle (5);
\draw (90, 290) node {$1$};
\draw  [fill={rgb, 255:red, 0; green, 0; blue, 0 }  ,fill opacity=1 ] (170,270) circle (5);
\draw (170, 290) node {$2$};
\draw [fill={rgb, 255:red, 0; green, 0; blue, 0 }  ,fill opacity=1 ] (250,270) circle (5);
\draw (250, 290) node {$3$};
\draw [fill={rgb, 255:red, 0; green, 0; blue,0 } ,fill opacity=1] (330,270) circle (5);
\draw (330, 290) node {$4$};

\end{tikzpicture}}

\newcommand{\graphE}[0]{
\begin{tikzpicture}[x=1pt,y=1pt,yscale=-0.5,xscale=0.5,baseline=-130pt, line width=1pt]

\draw [dash pattern={on 3pt off 2pt}] [-stealth]  (210,220)-- (245,265) ;
\draw [dash pattern={on 3pt off 3pt}]  [stealth-] (210,220)-- (325,265) ;
\draw [dash pattern={on 3pt off 3pt}]  [-stealth] (165, 270)--(210,220) ;

\draw [-stealth](170, 270)--(245,270) ;

\draw  [fill={rgb, 255:red, 0; green, 0; blue, 0 }  ,fill opacity=1 ] (170,270) circle (5);
\draw (170, 290) node {$1$};
\draw [fill={rgb, 255:red, 0; green, 0; blue, 0 }  ,fill opacity=1 ] (250,270) circle (5);
\draw (250, 290) node {$2$};
\draw [fill={rgb, 255:red, 0; green, 0; blue,0 } ,fill opacity=1] (330,270) circle (5);
\draw (330, 290) node {$3$};

\draw  (210,215) circle (5);
\draw (210, 200) node {$4$};

\end{tikzpicture}}

\newcommand{\graphF}[0]{
\begin{tikzpicture}[x=1pt,y=1pt,yscale=-0.5,xscale=0.5,baseline=-140pt, line width=1pt]

\draw [dash pattern={on 3pt off 2pt}] [-stealth]  (170,265)..controls (210,220)..(245,265) ;
\draw [dash pattern={on 3pt off 3pt}]  [-stealth] (170,275).. controls (210, 320).. (245,265) ;

\draw [dash pattern={on 3pt off 2pt}] [-stealth] (90, 270)--(165,270) ;
\draw [dash pattern={on 3pt off 2pt}] [-stealth](255, 270)--(325,270) ;

\draw  [fill={rgb, 255:red, 0; green, 0; blue, 0 }  ,fill opacity=1 ] (90,270) circle (5);
\draw (90, 290) node {$1$};
\draw  (170,270) circle (5);
\draw (170, 290) node {$2$};
\draw (250,270) circle (5);
\draw (250, 290) node {$3$};
\draw [fill={rgb, 255:red, 0; green, 0; blue,0 } ,fill opacity=1] (330,270) circle (5);
\draw (330, 290) node {$4$};

\end{tikzpicture}}

\begin{document} 
\title{Two graph homologies and the space of long embeddings}
\author{Leo Yoshioka\thanks{Graduate School of Mathematical Sciences, The University of Tokyo\newline\qquad e-mail:yoshioka@ms.u-tokyo.ac.jp}}
\date{October 2023}
\maketitle

\begin{abstract}
Graph homologies are powerful tools to compute the rational homotopy group of the space of long embeddings $\pi_{\ast}\text{Emb}(\mathbb{R}^j, \mathbb{R}^n)\otimes \mathbb{Q}$. 
Two graph homologies have been invented from two approaches to study the space of long embeddings: the hairy graph homology from embedding calculus, and BCR graph homology from configuration space integral. 
In this paper, we construct a monomorphism from the top hairy graph homology to the top BCR graph homology, though the latter graph homology is quite modified. This map and its left inverse are analogs of PBW isomorphism between $\mathcal{B}$ and $\mathcal{A}(S^1)$, the space of open Jacobi diagrams and the space of Jacobi diagrams on the unit circle, in the theory of Vassiliev knot invariants. 
\end{abstract}

\vskip\baselineskip

\tableofcontents

\section*{Introduction}
\addcontentsline{toc}{part}{Introduction}

A long embedding is an embedding of $\mathbb{R}^j$ into $\mathbb{R}^n$, which is standard outside a disk in $\mathbb{R}^j$. We write $\mathcal{K}_{n,j}$ for the space of long embeddings $\text{Emb}(\mathbb{R}^j, \mathbb{R}^n)$, and write $\overline{\mathcal{K}}_{n,j}$ for the homotopy fiber of the forgetful map $\mathcal{K}_{n,j} \rightarrow \text{Imm}(\mathbb{R}^j, \mathbb{R}^n)$ to the space of long immersions. The aim of this paper is to compare two combinatorial objects that are introduced to study $\overline{\mathcal{K}}_{n,j}$. 

In 2017, Fresse, Turchin and Willwacher \cite{FTW}, following Arone and Turchin \cite{AT1}\cite{AT2}, established a beautiful framework to compute the rational homotopy group $\pi_{\ast} \overline{\mathcal{K}}_{n,j} \otimes \mathbb{Q} $ in terms of combinatorial object called graph homology. 
More precisely, they showed the isomorphism between $\pi_{\ast} (\overline{\mathcal{K}}_{n,j})\otimes \mathbb{Q}$ and the graph homology $H_{\ast}(HGC^{n,j} )$, called hairy graph homology. This significant result is based on the theory of embedding calculus established by Goodwillie, Klein and Weiss \cite{GW}\cite{GKW}\cite{Wei} and also the theory of little disk operad. 

However, this deep framework is fully successful, so far, only for the range $n-j\geq3$.  Moreover, it seems difficult to know geometric generators of $\pi_{\ast} (\overline{\mathcal{K}}_{n,j}, \mathbb{Q})$  from generators of $H_{\ast}(HGC)$.
 
On the other hand, from 90's to 10's, Bott\cite{Bot}, Cattaneo, Rossi\cite{CR}, Sakai\cite{Sak} and Watanabe\cite{Wat} developed a framework to construct geometric  (co)cycles of  $\overline{\mathcal{K}}_{n,j}$ from another graph (co)homology. We call their graph homology BCR graph homology and write it as $H_{\ast}(GC^{n,j})$. To get the cocycles of $\overline{\mathcal{K}}_{n,j}$, they used the operation called configuration space integral. Fortunately, this framework is applicable to $n-j=2$. 

Surprisingly, BCR graph homology $H_{\ast}(GC)$ has very similar graphs to those of $H_{\ast}(HGC)$. However, although embedding calculus allows hairy graphs with any number of loops, configuration space integral was successful only for BCR graphs with zero or one loop. This caused much less information for $\pi_{\ast} (\overline{\mathcal{K}}_{n,j}, \mathbb{Q})$ or $H_{\ast} (\overline{\mathcal{K}}_{n,j}, \mathbb{Q})$ given by configuration space integral. 

In the author's previous paper \cite{Yos}, we have succeeded in conducting configuration space integral for some 2-loop BCR graphs, and in giving a new non-trivial  (co)cycle of the space of long embeddings (including $n-j =2$). However,  we haven't been able to give more examples, due to much more complicated combinatorics of $H_{\ast}(GC)$ than $H_{\ast}(HGC)$. 

 In this paper, we relate the two graph homologies $H_{\ast} (GC^{n,j})$ and $H_{\ast}(HGC^{n,j})$ at the top (non-degenerate) part.  We write the top part of $H_{\ast} (HGC)$ as $\mathcal{B}$ and write the top part of $H_{\ast} (GC)$ as $\mathcal{A}_{BCR}$, though $\mathcal{A}_{BCR}$ is quite modified as explained later. We focus on the case $n-j$ is even because the author is most interested in the case $n-j = 2$. 
 The following is our main result. 
 
 \begin{theorem}
 \label{main result 1 introduction}
 There is a monomorphism $\chi_{\ast}$ from $\mathcal{B}$ to $\mathcal{A}_{BCR}$. 
 \end{theorem}
 
 The proof of our main result is very similar to the proof by Bar-Natan\cite{Bar} of PBW (Poincar\'{e}--Birkhoff--Witt) isomorphisms between $\mathcal{B}$, the space of open Jacobi diagrams, and $\mathcal{A}(S_1)$, the space of Jacobi diagrams on the unit circle. As Bar-Natan showed another isomorphism between $\mathcal{A}(S_1)$ and the space of chord diagrams $\mathcal{A}^c(S_1)$, we have the following.  Let $\mathcal{A}^c_{BCR}$ be  the space of BCR chord diagrams defined in Definition \ref{defBCRchorddiagrams}.
 
 \begin{theorem}
 \label{main result 2introduction}
 There is an isomorphism $\iota_{\ast}$ from $\mathcal{A}^c_{BCR}$ to $\mathcal{A}_{BCR}$. 
 \end{theorem}

 It might be possible to apply our main result to give new nontrivial (co)cycles of  $\overline{\mathcal{K}}_{n,j}$ by the following steps: First, compute some part (two-loop part, for example) of $\mathcal{B}$ by relating it with (anti-)symmetric polynomials as in \cite{CCTW}. Next,  by Theorem \ref{main result 1 introduction}, we obtain some nontrivial subspace of  $\mathcal{A}_{BCR}$. This gives some non-trivial graph cocycle $H = \sum_i a_i \Gamma_i $. (See the next paragraph.) So we can perform configuration space integral
  \[
  \overline{I}(H) = \sum_i a_i \int_{\text{Conf}} \overline{\Omega} (\Gamma_i) 
 \]
associated with $H$ (with the correction term $\overline{c}$ in \cite{Yos}). Finally, conduct pairing argument in \cite{Yos} between the cocycle $\overline{I}(H)$ and a suitably-constructed cycle of $\overline{\mathcal{K}}_{n,j}$. Then  $\overline{I}(H)$ might detect the non-triviality of this cycle.

However, it should be emphasized that the (top) graph homology $\mathcal{A}_{BCR}$ is modified in this paper: $\mathcal{A}_{BCR}$ has no relation corresponding to contractions of chords. (See Definition \ref{defchordrelation}.) Hence, a graph cocycle $H$ constructed 
through a weight system $w$ on $\mathcal{A}_{BCR}$
\[
H = \sum_{\Gamma} \frac{w(\Gamma)}{\text{Aut}(\Gamma)} \Gamma
\]
may not be a genuine graph cocycle used in \cite{Yos}, because contractions of chords may not be canceled. One possible solution will be adding correction terms using the data of paths of immersions that $\overline{\mathcal{K}}_{n,j}$ has. We deal with the correction terms in the author's future work in preparation \cite{Yos2}.

The second result holds even if we impose relations on chords to the spaces $\mathcal{A}^c_{BCR}$ to $\mathcal{A}_{BCR}$. Let $\overline{\mathcal{A}}^c_{BCR}$ to $\overline{\mathcal{A}}_{BCR}$ be the quotient space.

\begin{theorem}
 \label{main result 3introduction}
 There is an isomorphism $\overline{\iota}$ from $\overline{\mathcal{A}}^c_{BCR}$ to $\overline{\mathcal{A}}_{BCR}$. 
 \end{theorem}
 
At present, the author does not know there is a monomorphism from $\mathcal{B}$ to the genuine graph homology $\overline{\mathcal{A}}_{BCR}$. But it holds for all the cases the author can compute $\overline{\mathcal{A}}_{BCR}$, including the simplest 2-loop case in \cite{Yos}. Note that the map is not necessarily surjective even in the case the target is the quotient space $\overline{\mathcal{A}}_{BCR}$.

\begin{question}
Is the map $\overline{\chi}_{\ast}$ from $\mathcal{B}$ to $\overline{\mathcal{A}}_{BCR}$ injective?
\end{question} 

As a final remark, we go back to Arone and Turchin's paper \cite{AT2}. In \cite{AT2}, the graph homology $H_{\ast}(GC)$, which is written there as $H_{\ast}(\mathcal{E}^{j,n}_{\pi})$, is constructed from the homology of a derived mapping space between $\Omega$ modules:
\[
H_{\ast}\left(\underset{\Omega}{\text{Rmod}}^h (\widetilde{H}_{\ast}(S^{m\bullet}), \overline{H}_{\ast}(C_{\bullet}(\mathbb{R}^n))\right)
\]
by an injective resolution of the target. In \cite{AT2}, another graph homology $H_{\ast}(\mathcal{K}^{j,n}_{\pi})$ is introduced by a projective resolution of the source. By taking both projective and injective resolutions, one obtains one more graph homology $H_{\ast}(\mathcal{Y}^{j,n}_{\pi})$, which does not explicitly appear in \cite{AT2}. Since the origin is the same, the three graph homologies are isomorphic. It would be interesting to see directly the differences between $H_{\ast}(\mathcal{Y}^{j,n}_{\pi})$ and $H_{\ast}(GC)$; the duals of $H_{\ast}(GC)$ and $H_{\ast}(\mathcal{Y}^{j,n}_\pi)$ have almost the same graphs. A similar question arises for differences between $\mathcal{A}^c_{BCR}$ and the top part of $H_{\ast}(\mathcal{K}^{j,n}_\pi)$. See Remark \ref{remarkfromEC2}.
 
This paper is organized as follows. In Section \ref{Definition of graph homologies}, we define three spaces of graphs: $\mathcal{B}$, $\mathcal{A}_{BCR}$ and $\mathcal{A}^c_{BCR}$. Our main result is stated in Section \ref{Main Result}. After introducing several notations for proof in Section \ref{Preliminary for proof of Main Result}, we show our main result in Section \ref{Proof of Main Result}. 

\section*{Acknowledgement}
The author is deeply grateful to Victor Turchin for explaining to the author a lot about their work on graph homology. 
The author would like to thank Tadayuki Watanabe for motivating and inspiring discussion with the author. 
This paper was written during the author's three-month stay at Kansas State University. The author appreciates their hospitality. This research was supported by Forefront Physics and Mathematics Program to Drive Transformation (FoPM), a World-leading Innovative Graduate Study (WINGS) Program, The University of Tokyo.

\section{Definition of graph homologies}
\label{Definition of graph homologies}

\subsection{The modified BCR graph homology $\mathcal{A}_{BCR}$}
\label{defBCRgraphhomology}
Here, we define the spaces $\mathcal{A}_{BCR}$ and $\overline{\mathcal{A}}_{BCR}$ of BCR graphs. This space is almost equivalent to the dual of the top graph cohomology $H^0(GC)$ defined in \cite{Yos} as $H^0(\mathcal{D})$. 
However, we add an extra condition (the last condition in the following definition)  to the graphs. Moreover, we exclude some relations on chords. 
\begin{definition}[BCR graphs]
A BCR graph is a connected graph that satisfies the following. 
\begin{itemize}
\item There are (at most) two types of vertices: external vertices drawn in black and internal vertices drawn in white. The graph must have at least one external vertex. 
\item There are (at most) two types of edges: dashed edges and solid edges.
\item Each internal vertex has exactly three dashed edges and no solid edge. 
\item Each external vertex has exactly one dashed edge and at most two solid edges. 

\item The restriction to solid edges of the graph forms a disjoint sum of broken lines such as
\begin{tikzpicture}[x=0.5pt,y=0.5pt,yscale=0.5,xscale=0.5, baseline=0pt, line width = 1pt] 
\draw  [dash pattern = on 3pt off 2pt] (0, 0)-- (0,50);
\draw  [dash pattern = on 3pt off 2pt] (50, 0)-- (50,50);
\draw  [dash pattern = on 3pt off 2pt] (100, 0)-- (100,50) ;
\draw [dash pattern = on 3pt off 2pt]  (150, 0)--(150,50);
\draw  [dash pattern = on 3pt off 2pt] (200, 0)--(200,50);
\draw [dash pattern = on 3pt off 2pt]  (250, 0)--(250,50);

\draw  (0, 0)-- (100,0);
\draw  (150, 0)-- (200,0);

\draw [fill = black] (0, 0) circle (7);
\draw [fill = black] (50, 0) circle (7);
\draw [fill = black]  (100, 0) circle (7);

\draw [fill = black] (150, 0) circle (7);
\draw [fill = black](200, 0) circle (7);

\draw [fill = black]  (250, 0) circle (7);
\end{tikzpicture}.
In particular, neither double solid edges
\begin{tikzpicture}[x=1pt,y=1pt,yscale=1,xscale=1, baseline=-1pt, line width =1pt]
\draw  [fill={rgb, 255:red, 0; green, 0; blue, 0 }  ,fill opacity=1 ] (-1,0) circle (1) ;
\draw  [fill={rgb, 255:red, 0; green, 0; blue, 0 }  ,fill opacity=1 ] (50,0) circle (1) ;
\draw  (0,0) .. controls (20,5) and (30,5)  .. (50,0);
\draw (0,0) .. controls (20,-5) and (30,-5)  .. (50,0);
\end{tikzpicture}
nor small solid loops
\begin{tikzpicture}[x=1pt,y=1pt,yscale=1,xscale=1, baseline=1pt, line width = 1pt]
\draw (0,0) .. controls (-20,20) and (20,20)  .. (0,0);
\draw [fill = black] (0,-1) circle (1.5);
\end{tikzpicture}
are allowed.
\end{itemize}
\end{definition}

\begin{rem}
In Definition \ref{defBCRgraphhomology}, small dashed loops
\begin{tikzpicture}[x=1pt,y=1pt,yscale=1,xscale=1, baseline=1pt, line width = 1pt]
\draw [dash pattern = on 3pt off 2pt] (0,0) .. controls (-20,20) and (20,20)  .. (0,0);
\draw(0,-1) circle (1.5);
\end{tikzpicture},
double dashed edges
\begin{tikzpicture}[x=1pt,y=1pt,yscale=1,xscale=1, baseline=-1pt, line width =1pt]
\draw  (-1,0) circle (1) ;
\draw  (52,0) circle (1) ;
\draw  [dash pattern = on 3pt off 2pt] (0,0) .. controls (20,5) and (30,5)  .. (50,0);
\draw [dash pattern = on 3pt off 2pt] (0,0) .. controls (20,-5) and (30,-5)  .. (50,0);
\end{tikzpicture}
and multiple edges
\begin{tikzpicture}[x=1pt,y=1pt,yscale=1,xscale=1, baseline=-1pt, line width =1pt]
\draw  [fill = black]  (-1,0) circle (1) ;
\draw  [fill = black] (50,0) circle (1) ;
\draw  [dash pattern = on 3pt off 2pt] (0,0) .. controls (20,5) and (30,5)  .. (50,0);
\draw (0,0) .. controls (20,-5) and (30,-5)  .. (50,0);
\end{tikzpicture}
are allowed. However, some of them vanish by relations we later define.
\end{rem}

\begin{definition}
\label{defcolororientation}
We define two ways of coloring  BCR graphs: (I) odd case and (II) even case.
\begin{itemize}
\item [(I)] We orient each edge and label the vertices from $1$ to $n$ if the number of vertices is $n$.
\item [(II)]  We number the edges from $(1)$ to $(m)$ if the number of edges is $m$. 
\end{itemize}
We say two colorings of the isomorphic underlying graph are the same (resp. opposite) orientation if they are transferred from each other by even (resp. odd) permutation. 
\end{definition}

\begin{definition}
We write $\mathcal{D}$ for the vector space over $\mathbb{Q}$ generated by colored BCR graphs. 
 \end{definition}

\begin{definition}
We write $\mathcal{A}_{BCR}$ for the quotient vector space of $\mathcal{D}$ by IHX, STU, and orientation relations. We define STU and IHX relations below.
\end{definition}

\begin{definition}
We define IHX relations of colored BCR graphs as in Figure \ref{IHXrelation} below. In Figure \ref{IHXrelation}, graphs $D_1, D_2, D_3$ are allowed to have any colors so that contractions of the middle edges give three isomorphic graphs with the same orientation  (after transposition to the left-hand side). See \cite {Yos} for the sign rule for contraction of edges.
\begin{figure} [htpb]
\begin{center}
\begin{tikzpicture}[x=0.75pt,y=0.75pt,yscale=0.4,xscale=0.4, baseline=30pt, line width = 1pt] 
\draw [dash pattern={on 5pt off 4pt}]  (95,155)--(0,200); 
\draw [dash pattern={on 5pt off 4pt}]  (105,155)--(200,200); 
\draw [dash pattern={on 5pt off 4pt}]  (100,55)--(100,145); 
\draw (135, 100) node {};
\draw [dash pattern={on 5pt off 4pt}]  (95,45)--(0,0); 
\draw [dash pattern={on 5pt off 4pt}]  (105,45)--(200,0); 
\draw [color={rgb, 255:red, 0; green, 0; blue, 0 }  ,draw opacity=1 ] (100, 50) circle (7);
\draw [color={rgb, 255:red, 0; green, 0; blue, 0 }  ,draw opacity=1 ] (100, 150) circle (7);
\draw (100, -50) node {$D_1$};
\end{tikzpicture}
=
\begin{tikzpicture}[x=0.75pt,y=0.75pt,yscale=0.4,xscale=0.4, baseline=30pt, line width = 1pt] 
\draw [dash pattern={on 5pt off 4pt}]  (45,105)--(0,200);
\draw [dash pattern={on 5pt off 4pt}]  (155,105)--(200,200); 
\draw [dash pattern={on 5pt off 4pt}]  (55,100)--(145,100); 
\draw (100, 130) node {};
\draw [dash pattern={on 5pt off 4pt}]  (45,95)--(0,0); 
\draw [dash pattern={on 5pt off 4pt}]  (155,95)--(200,0); 
\draw [color={rgb, 255:red, 0; green, 0; blue, 0 }  ,draw opacity=1 ] (50, 100) circle (7);
\draw [color={rgb, 255:red, 0; green, 0; blue, 0 }  ,draw opacity=1 ] (150, 100) circle (7);
\draw (100, -50) node {$D_2$};
\end{tikzpicture}
-
\begin{tikzpicture}[x=0.75pt,y=0.75pt,yscale=0.4,xscale=0.4, baseline=30pt, line width = 1pt] 
\draw [dash pattern={on 5pt off 4pt}]  (145,105)--(0,200);
\draw [dash pattern={on 5pt off 4pt}]  (55,105)--(200,200); 
\draw [dash pattern={on 5pt off 4pt}]  (55,100)--(145,100); 
\draw (100, 70) node {};
\draw [dash pattern={on 5pt off 4pt}]  (45,95)--(0,0); 
\draw [dash pattern={on 5pt off 4pt}]  (155,95)--(200,0); 
\draw [color={rgb, 255:red, 0; green, 0; blue, 0 }  ,draw opacity=1 ] (50, 100) circle (7);
\draw [color={rgb, 255:red, 0; green, 0; blue, 0 }  ,draw opacity=1 ] (150, 100) circle (7);
\draw (100, -50) node {$D_3$};
\end{tikzpicture}.

\end{center}
\caption{IHX relation}
\label{IHXrelation}
\end{figure}
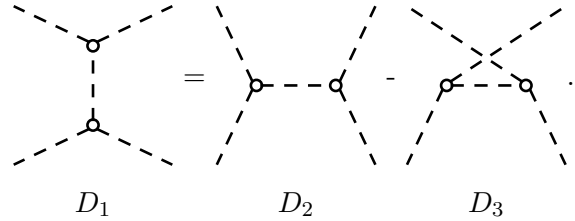
\end{definition}

\begin{rem}
As a consequence of an IHX relation, a small dashed loop at an internal vertex that is connected to another internal vertex vanishes: 
\[
\begin{tikzpicture}[x=1pt,y=1pt,yscale=2,xscale=2, baseline=1pt, line width = 1pt]
\draw [dash pattern = on 5pt off 4pt] (0,0) .. controls (-20,20) and (20,20)  .. (0,0);
\draw (0,-1) circle (1.5);
\draw (0,-10) circle (1.5);
\draw [dash pattern = on 5pt off 4pt] (0,-2.5)-- (0,-8.5);
\end{tikzpicture}=0.
\]
\end{rem}

\begin{definition}
We define STU relations (I)(II) of colored BCR graphs as in Figure \ref{STUrelation} below. In Figure \ref{STUrelation}, graphs $D_i$ are allowed to have any colors so that contractions of the middle edges (circled in red) give isomorphic graphs with the same orientation. In the case (I), the bottom vertex of the graph $D_1$ has one or two solid edges. In the case (II), the bottom vertex of the graph $D_1$ has no solid edge.

\begin{figure}[htpb]
\begin{center}
\includegraphics[width = 11cm]{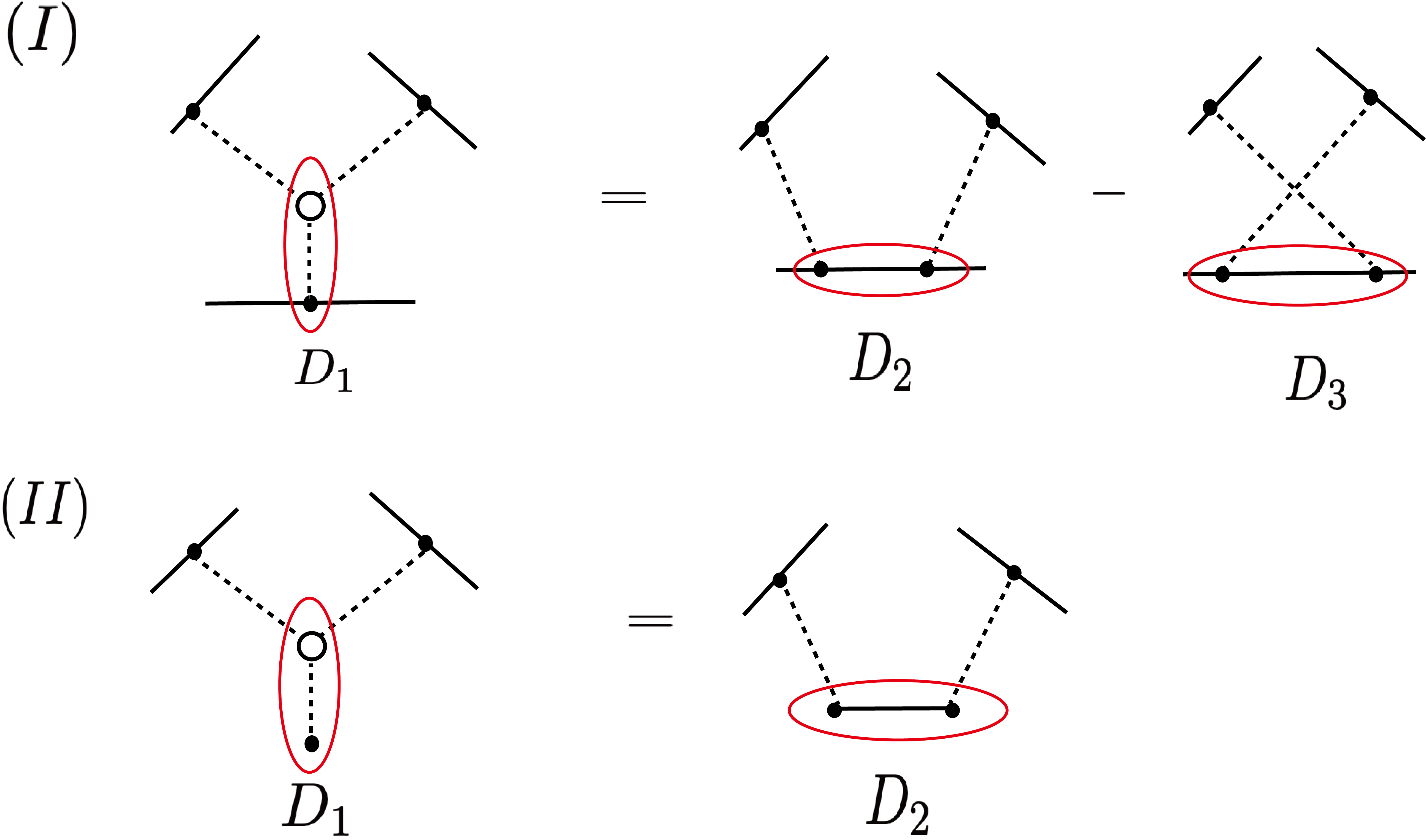}
\caption{STU relations (I)(II)}
\label{STUrelation}
\end{center}
\end{figure}

\end{definition}

\begin{rem}
\label{smallloop2}
As a consequence of an STU relation in the odd case (Definition \ref{defcolororientation}), a small dashed loop at an internal vertex that is connected to an external vertex vanishes: 
\[
\begin{tikzpicture}[x=1pt,y=1pt,yscale=2,xscale=2, baseline=1pt, line width = 1pt]
\draw [dash pattern = on 5pt off 4pt] (0,0) .. controls (-20,20) and (20,20)  .. (0,0);
\draw (0,-1) circle (1.5);
\draw [fill= black](0,-10) circle (1.5);
\draw [dash pattern = on 5pt off 4pt] (0,-2.5)-- (0,-8.5);
\end{tikzpicture}=0\quad (\text{odd case}).
\]
\end{rem}

\begin{rem}
As a consequence of an orientation relation in the even case, a double dashed edge vanishes. 

\[
\begin{tikzpicture}[x=1pt,y=1pt,yscale=2,xscale=2, baseline=-1pt, line width = 1pt ]
\draw  (-1,0) circle (1.5) ;
\draw  (52,0) circle (1.5) ;
\draw  [dash pattern = on 5pt off 4pt] (0,0) .. controls (20,5) and (30,5)  .. (50,0);
\draw  [dash pattern = on 5pt off 4pt] (0,0) .. controls (20,-5) and (30,-5)  .. (50,0);
\end{tikzpicture}\ =0 \quad  (\text{even case}).
\]
\end{rem}

\begin{definition}
We write $\overline{\mathcal{A}}_{BCR}$ for the quotient vector space of $\mathcal{D}$ by IHX, STU, orientation relations and chord relations. We define chord relations below.
\end{definition}

\begin{definition}
\label{defchordrelation}
We define chord relations (I)(II) of colored BCR graphs as in Figure \ref{chordrelation} below. In Figure \ref{chordrelation}, graphs $D_i$ are allowed to have any colors so that contractions of the middle edges give isomorphic graphs with the same orientation. 
Here, We only use chord relations in which all the graphs $D_i$ are modified BCR graphs. Namely, if the contracted graph has a loop consisting of only solid edges and external vertices, we do not consider the relation as a chord relation. In the case (I), the bottom vertex of the graph $D_1$ has one or two solid edges. In the case (II), the bottom vertex of the graph $D_1$ has no solid edge.

\begin{figure}[htpb]
\begin{center}
\includegraphics[width = 10cm]{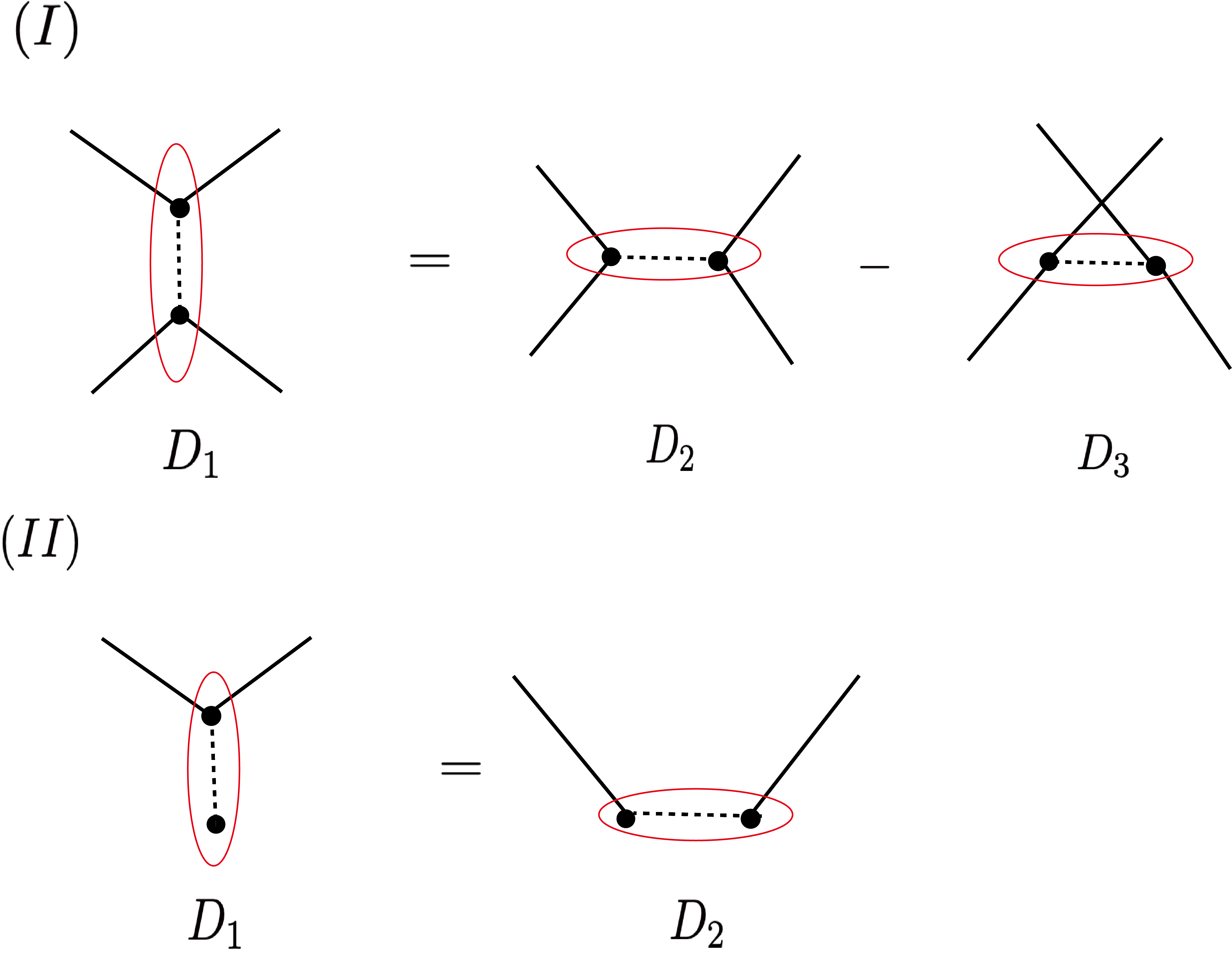}
\caption{Chord relations (I)(II)}
\label{chordrelation}
\end{center}
\end{figure}

\end{definition}

\begin{rem}
We exclude contractions of chords that yield loops consisting of only solid edges for the following reason: if the dimension $j$ of the space of embedding $\mathcal{K}_{n,j}$ is large enough comparing the order (size) $k$, more precisely $j\geq 2k$, the integral by such graphs vanish by dimensional reason. This argument is introduced in Section 8 of \cite{Yos}. Our second result \ref{main result 2introduction} holds even if we admit graphs with solid loops and add chord relations that include such graphs. 

\end{rem}

\subsection{The hairy graph homology $\mathcal{B}$}

We define the space $\mathcal{B}$ of hairy graphs. 
The space $\mathcal{B}$ is an analog of the space of open Jacobi diagrams modulo IHX and AS relations (and they coincide in the odd case). 

\begin{definition}[Hairy graphs]
A hairy graph is a connected graph that satisfies the following. 
\begin{itemize}
\item There are (at most) two types of vertices: external vertices drawn in black and internal vertices drawn in white. The graph must have at least one external vertex.
\item There is only one type of edges, dashed.
\item Each external vertex has exactly one dashed edge. Each internal vertex has exactly three dashed edges. 
\end{itemize}

We can define color and orientation of hairy graphs by those of BCR graphs.
\end{definition}

\begin{definition}
\label{defofhairygraphhomology}
We write $\mathcal{C}$ for the vector space over $\mathbb{Q}$ generated by colored hairy graphs.
We define the space $\mathcal{B}$ as the quotient space of $\mathcal{C}$ by orientation relations and IHX relations. In the odd case, we add the following relation (See Remark \ref{smallloop2}). 
\[
\begin{tikzpicture}[x=1pt,y=1pt,yscale=2,xscale=2, baseline=1pt, line width = 1pt]
\draw [dash pattern = on 5pt off 4pt] (0,0) .. controls (-20,20) and (20,20)  .. (0,0);
\draw (0,-1) circle (1.5);
\draw [fill= black](0,-10) circle (1.5);
\draw [dash pattern = on 5pt off 4pt] (0,-2.5)-- (0,-8.5);
\end{tikzpicture}=0\quad (\text{odd case}).
\]

\end{definition}

\begin{rem}
\label{remarkfromEC0}
In \cite{AT2}, Arone and Turchin introduced a graph complex $\mathcal{E}^{j,n}_{\pi}$, which is shown to be isomorphic to the rational homotopy group $\pi_{\ast} (\mathcal{K}_{n,j}) \otimes \mathbb{Q}$ for $n-j \geq 3$ by Fresse, Turchin and Willwacher \cite{FTW}. The top part of $H_{\ast} (\mathcal{E}^{j,n}_{\pi}) $ is isomorphic to $\mathcal{B}$. In the odd case, $\mathcal{B}$ is isomorphic to the space of hairy graphs modulo IHX and AS relations \cite{Bar}. 
\end{rem}

\subsection{The space of BCR chord diagrams $\mathcal{A}^c_{BCR}$}

Next, we define the spaces of BCR chord diagrams $\mathcal{A}^c_{BCR}$ and $\overline{\mathcal{A}}^c_{BCR}$.
These spaces are analogs of the space of chord diagrams  $\mathcal{A}^c(S^1)$, which appears in the theory of Vassiliev knot invariants.

\begin{definition}[BCR chord diagrams]
\label{defBCRchorddiagrams}
A BCR chord diagram is a connected graph that satisfies the following. 
\begin{itemize}
\item There is only one type of vertices: external vertices drawn in black. 
\item There are two types of edges, solid and dashed.
\item Each external vertex has exactly one dashed edge and at most two solid edges. 
\item The restriction to solid edges of the graph forms a disjoint union of broken lines. See Definition \ref{defBCRgraphhomology}.
\end{itemize}

We can define color and orientation of BCR chord diagrams by those of BCR graphs.
\end{definition}

\begin{definition}
We write $\mathcal{D}^c$ for the vector space over $\mathbb{Q}$ generated by colored BCR chord diagrams.
We define $\mathcal{A}^c_{BCR}$  as the quotient space of  $\mathcal{D}^c$ by orientation relations and $4T$ relations, which are defined below.
\end{definition}

\begin{definition}
We define $4T$ relations (I)(II)(III) of colored BCR chord diagrams as in Figure \ref{4Trelation} below. In Figure \ref{4Trelation}, graphs $D_i$ are allowed to have any colors so that $S = D_1 - D_2 $ and $S = D_3 - D_4 $ are both STU relations for some colored BCR graph $S$. The graph $S$ is uniquely determined after orientation relations in $\mathcal{A}_{BCR}$.

 \begin{figure}[htpb]
\begin{center}
\includegraphics[width = 10cm]{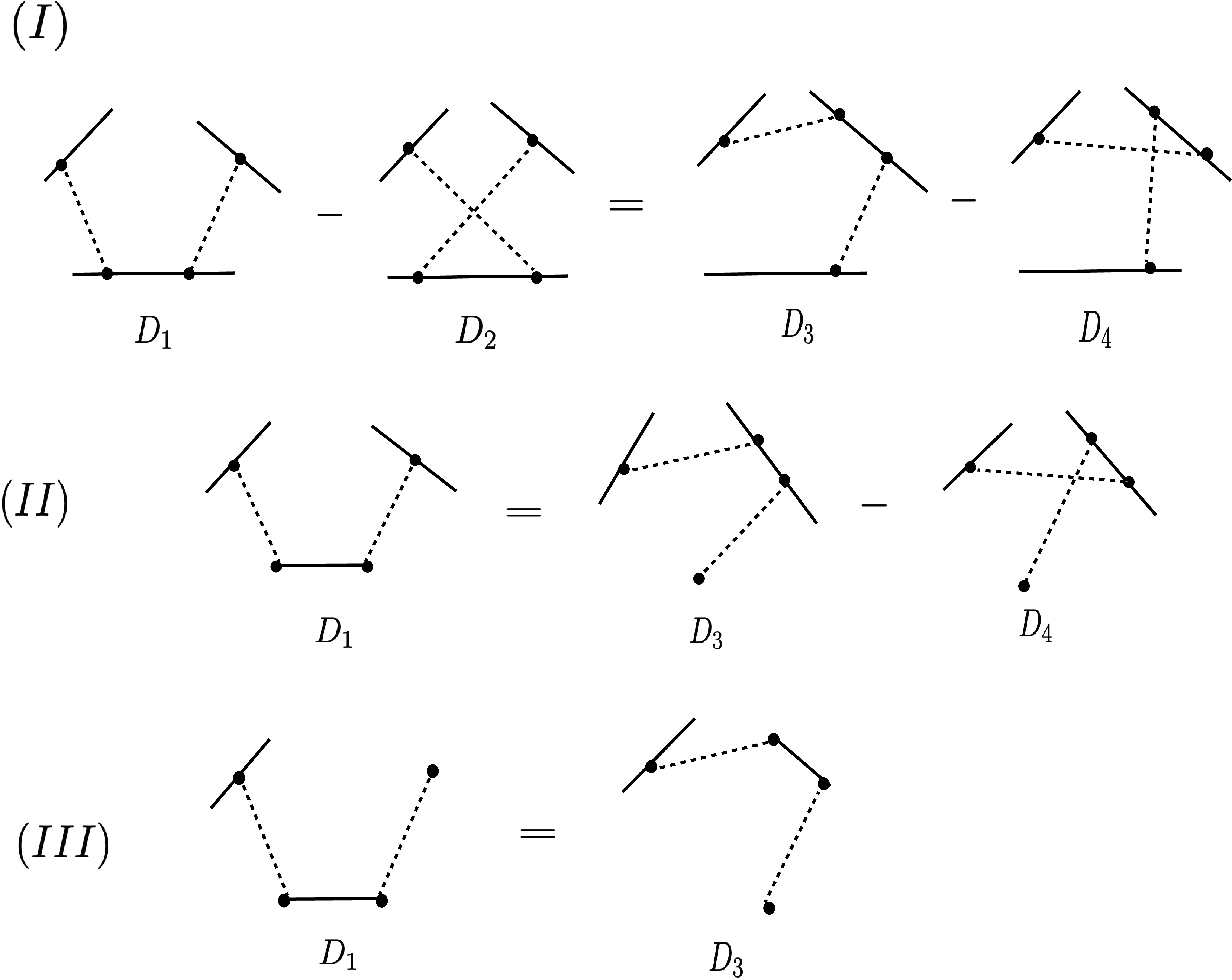}
\caption{4T relations (I)(II)(III)}
\label{4Trelation}
\end{center}
\end{figure}

\end{definition}

\begin{definition}
We define $\overline{\mathcal{A}}^c_{BCR}$  as the quotient space of  $\mathcal{D}^c$ by orientation relations, $4T$ relations and chord relations. 
Here, we only use chord relations in which all the graphs $D_i$ are BCR chord diagrams. 
\end{definition}

\begin{rem}
\label{remarkfromEC2}
In \cite{AT2}, Arone and Turchin introduced a complex $\mathcal{K}^{j,n}_{\pi}$ that is quasi-isomorphic to $\mathcal{E}^{j,n}_{\pi}$. Though $\mathcal{K}^{j,n}_{ \pi}$ is a little difficult to describe, its dual complex $HH^{j,n}_{\pi}$ is described in terms of graphs \cite{AT2}. The top (non-degenerate) part of $HH^{j,n}_{\pi}$  is generated by BCR chord diagrams. However, they allow not only broken lines but also trees as the restriction to solid edges. They also impose Arnold relations that come from relations in  $H^{\ast}(C_k(\mathbb{R}^m))$ and $H^{\ast}(C_k(\mathbb{R}^n))$. As a consequence of Arnold relations (with respect to solid edges), trees of solid edges can be transformed into a linear combination of broken lines. Our $4T$ relations correspond to the differential $d$ of $HH^{j,n}_{\pi}$ (with  Arnold relations with respect to dashed edges). So far, the author finds nothing related to chord relations in the complex $HH_{j,n}^{\pi}$.
\end{rem}

\section{Main Result}
\label{Main Result} 

\subsection{The map $\chi_{\ast}: \mathcal{B}  \rightarrow \mathcal{A}_{BCR}$}
\begin{prop} \label{Main Result 01}
We define the map $\chi: \mathcal{C} \rightarrow \mathcal{D}$ by the inclusion. Then this map induces $\chi_{\ast}: \mathcal{B} \rightarrow \mathcal{A}_{BCR}$. 
\end{prop}

\begin{theorem}\label{Main Result 1}
The map $\chi_{\ast}: \mathcal{B}  \rightarrow \mathcal{A}_{BCR}$ is a monomorphism.  
\end{theorem}

\subsection{The map $\iota_{\ast}: \mathcal{A}^c \rightarrow \mathcal{A}_{BCR}$}
\begin{prop} \label{Main Result 02}
We define the map $\iota: \mathcal{D}^c \rightarrow \mathcal{D}$ by the inclusion. Then this map induces the maps  $\iota_{\ast}: \mathcal{A}^c_{BCR} \rightarrow \mathcal{A}_{BCR}$ and $\overline{\iota}_{\ast}: \overline{\mathcal{A}}^c_{BCR} \rightarrow \overline{\mathcal{A}}_{BCR}$. 
\end{prop}

\begin{theorem}\label{Main Result 2}
Both of the maps  $\iota_{\ast}: \mathcal{A}^c_{BCR} \rightarrow \mathcal{A}_{BCR}$ and $\overline{\iota}_{\ast}: \overline{\mathcal{A}}^c_{BCR} \rightarrow \overline{\mathcal{A}}_{BCR}$ are isomorphisms.  
\end{theorem}

\section{Preliminary for proof of Main Result }
\label{Preliminary for proof of Main Result}
Here, we introduce notations used in Section \ref{Proof of Main Result}.
\begin{definition}
A solid component of a BCR graph $D$ is a component of solid broken lines. That is, a solid component consists of black vertices $v_{1}, v_{2}, \dots, v_{n}$ and solid edges $e_i = (v_{i}, v_{i+1})$ for any $i = 1, \dots, n-1$, where $v_1, v_2, \dots, v_n$ have no other solid edges. 
The solid type $T(D)$ of $D$ is the collection (in descending order) of lengths of solid components. For example, the graph in Figure \ref{solid type} has the solid type $(2,2)$. 
\end{definition}

\begin{figure}[htpb]
\begin{center}
\includegraphics[width = 5cm]{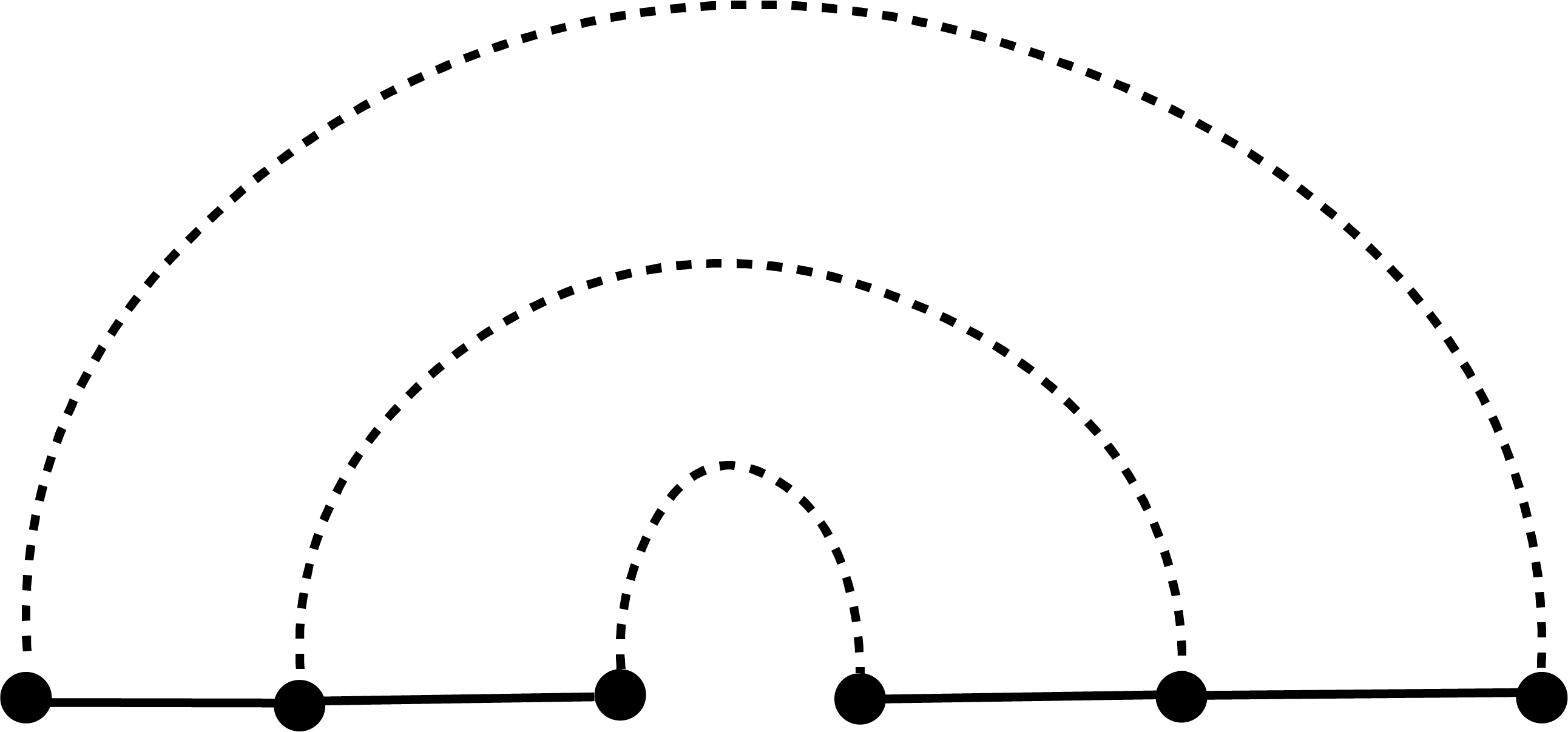}
\caption{A graph with the solid type (2,2) }
\label{solid type}
\end{center}
\end{figure}

\begin{notation}
Let $\mathcal{C}_k$ be the subspace of  $\mathcal{C}$ spanned by colored hairy graphs with at most $k$ external vertices. 
Let $\mathcal{D}_k$ be the subspace of  $\mathcal{D}$ spanned by colored BCR graphs with at most $k$ external vertices. Clearly $\chi: \mathcal{C} \rightarrow \mathcal{D}$ restricts to $\chi: \mathcal{C}_k \rightarrow \mathcal{D}_k$. 
Let $\mathcal{I}_k$ be the subspace of $\mathcal{D}_k$ spanned by STU relations of graphs with at most $k$ external vertices. 

\end{notation}

\begin{notation}
Let $D$ be a colored BCR graph with $k$ external vertices. Assume that solid components of $D$ are in a row as in Figure \ref{permutationofgraph}. 
For $\pi \in \mathfrak{S}_k$, let $(\pi D)_u$  be the colored graph obtained from $D$ by permutating dashed edges so that $i$th edge from the left goes to the $\sigma(i)$th. 
Then we set
\[
\pi D = \text{sign}(\pi) (\pi D)_u \in \mathcal{D}.
\]
\end{notation}
Let $\mathfrak{S}_k(D)$ be the subset of $\mathfrak{S}_k$ which consists of permutations that preserve solid components. That is, if $\pi \in \mathfrak{S}_k$, the $i$th and $\pi(i)$th edges must be connected to the same solid component. If the solid type of $D$ is $T(D) = (i_1,\dots, i_s)$, $\mathfrak{S}_k(D)$ is isomorphic to 
\[
\prod_{l=1}^{s} \mathfrak{S}_{i_l}.
\]

\begin{figure}[htpb]
\begin{center}
\includegraphics[width = 10cm]{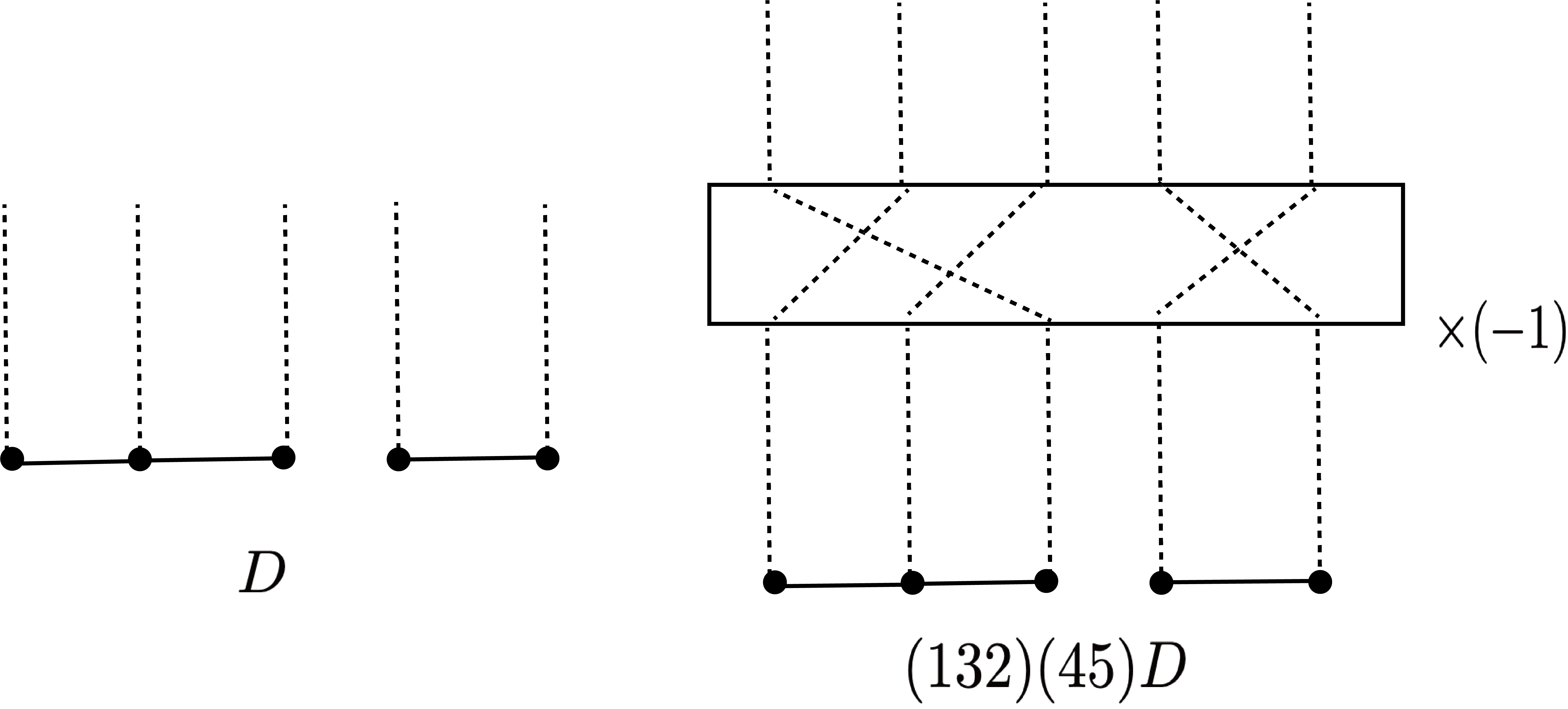}
\caption{A permutation of $D$}
\label{permutationofgraph}
\end{center}
\end{figure}

\begin{notation}
Let $D$ be a colored BCR graph with $k$ external vertices. Let $U_i$ be a transposition in $\mathfrak{S}_k(D)$. 
We define $S_i D$ as a colored BCR graph with $k-1$ vertices so that the equation
\[
S_i D = D - U_i D
\]
is an STU relation. The graph $S_i D$ is determined uniquely after orientation relation in $\mathcal{D}$. See Figure \ref{STUrelation2}.

\begin{figure}[htpb]
\begin{center}
\includegraphics[width = 10cm]{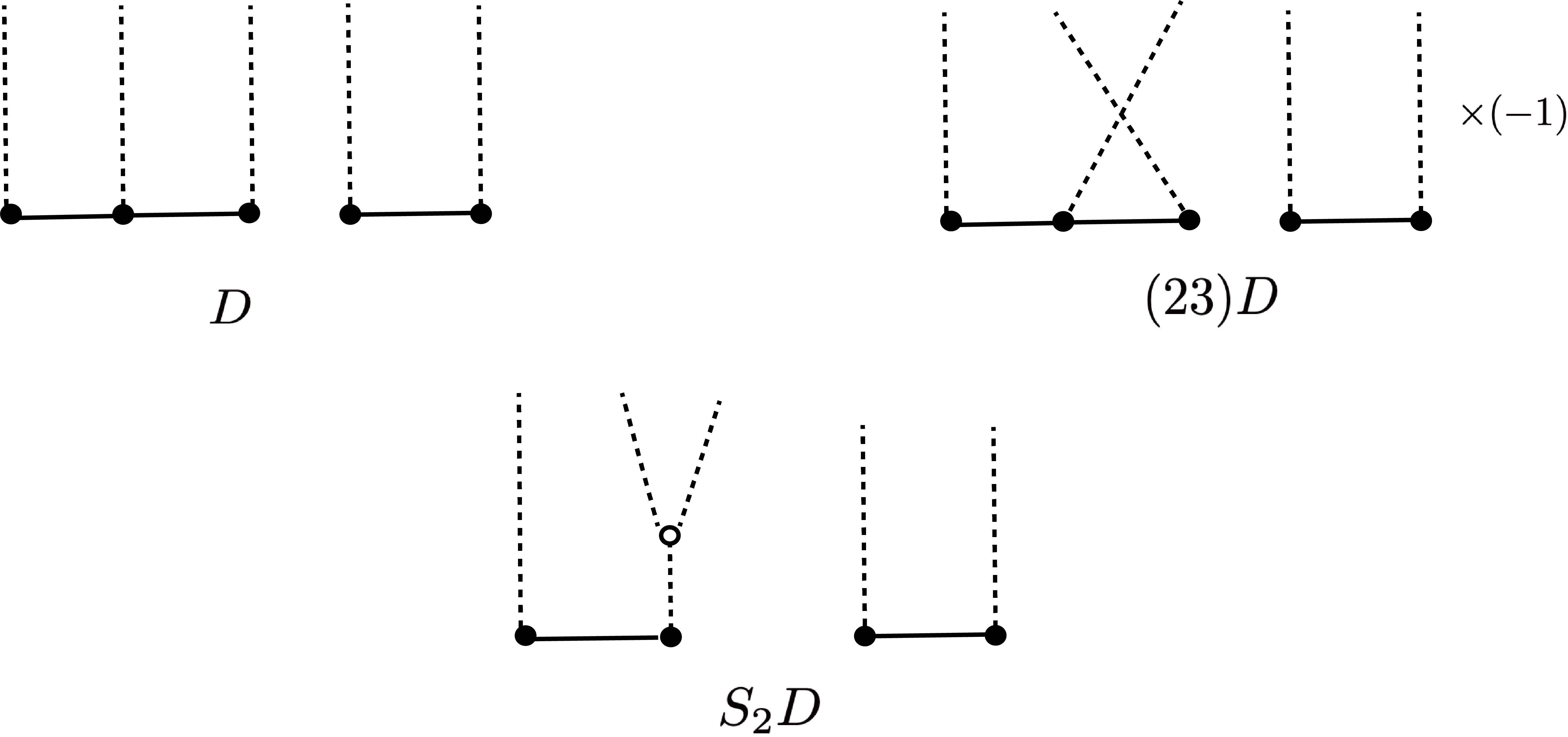}
\caption{Example of $S_i D$}
\label{STUrelation2}
\end{center}
\end{figure}

\end{notation}

\section{Proof of Main Result}
\label{Proof of Main Result}

The proof of Proposition \ref{Main Result 01} and \ref{Main Result 02} is trivial by the definitions of their relations. 
We prove Theorem \ref{Main Result 1} and Theorem \ref{Main Result 2}  in the following. The proof is very similar to the proof of \cite[Theorem 8]{Bar} and \cite[Theorem 6]{Bar}.

\subsection{$\chi_{\ast}: \mathcal{B} \rightarrow \mathcal{A}_{BCR}$ is a monomorphism}

\begin{proof}[Proof of Theorem \ref{Main Result 1}] 
We construct the left inverse $\sigma_{\ast}$ of the map $\chi_{\ast}$. 
We construct inductively the map $\sigma_k: \mathcal{D}_k  \rightarrow \mathcal{B}$ which satisfies the following. 
\begin{itemize}
\item [($\Sigma 1$)] $\sigma_k \circ \chi_k: \mathcal{C}_k \rightarrow \mathcal{B}$  is the projection. 
\item [($\Sigma 2$)] $\sigma_k$(orientation relation) $=0. $
\item [($\Sigma 3$)] $\sigma_k$(IHX relation) $=0. $ 
\item [($\Sigma 4$)] $\sigma_k(\mathcal{I}_k) =0. $
\end{itemize}
First, we define $\sigma_{1}: \mathcal{D}_1 \rightarrow  \mathcal{B}$ as the identity. Clearly $\sigma_{1}$ satisfies ($\Sigma 1$)($\Sigma 2$)($\Sigma 3$)($\Sigma 4$). 
Next we assume that $\sigma_{k-1}: \mathcal{D}_{k-1}  \rightarrow  \mathcal{B}$ is constructed so that it satisfies ($\Sigma 1$)($\Sigma 2$)($\Sigma 3$)($\Sigma 4$).
Then we define $\sigma_{k}: \mathcal{D}_{k}  \rightarrow \mathcal{B}$ by
\begin{equation*}
\sigma_{k}(D) =  \frac{1}{|\mathfrak{S}_k(D)|} \sum_{\pi \in \mathfrak{S}_k(D)} \sigma_{k-1} (\Gamma_D(\pi)) 
\end{equation*}
for a graph $D$ with $k$ external vertices.\footnote{Remember that Bar-Natan's original construction was $\sigma_k (D) = \frac{1}{k!} (D^{cc} + \sum_{\pi_\in \mathfrak{S}_k} \Lambda_D(\pi))$. We do not use $D^{cc}$ part.}
Here $\Gamma_D(\pi)$ is an element of $\mathcal{D}_{k-1}$ determined uniquely after orientation relations. In $\mathcal{A}_{BCR}$, the element $\Gamma_D(\pi)$ satisfies
\[
D - \pi D = \Gamma_D(\pi).
\]
 
We explain the construction of $\Gamma_{D} (\pi)$. 
If $\pi$ is the elementary transposition $U_{i} = (i,i+1) \in \mathfrak{S}_k(D)$, we set $\Gamma_{D} (\pi = U_i) = S_i D$. For a general permutation $\sigma \in \mathfrak{S}_k(D)$, write $\sigma$ as the product of elementary transpositions of $\mathfrak{S}_k(D)$: $\pi = U_{i_1}\dots U_{i_n }$. 
Then 
\[
\Gamma_D (\pi) = \sum_{l=1}^n  S_{i_l} U_{i_{l+1}}\dots U_{i_n}\Gamma. 
\]

\begin{example}
Let $D = \graphD $, $D^{\prime} = \graphDb$, $E = \graphE$, $F = \graphF$ respectively. Then we have

\begin{equation*}
\Gamma_D(\pi) =
\begin{cases}
E & \text{when\ } \pi = (1\ 2), (2\ 3), (1\ 2\ 3), (1\ 3\ 2) \in \mathfrak{S}_3,  \\
0 & \text{when\ } \pi = id, (1\ 3) \in \mathfrak{S}_3. \\
\end{cases}
\end{equation*}
So we have 
\begin{align*}
\sigma_4(D) &= \frac{1}{6} \sum_{\pi \in \mathfrak{S}_3} \sigma_3(\Gamma_D(\pi))= \frac{2}{3} \sigma_3(E) = \frac{2}{3} \sigma_2 (F) = \frac{2}{3} F .
\end{align*}
Similarly, we have
\[
\sigma_4(D^{\prime})  = \frac{1}{3} \sigma_3 (E) = \frac{1}{3} \sigma_2 (F) = \frac{1}{3} F
\]
Note that $F = E = D+D^{\prime}$ are STU relations.

\end{example}

We must show that $\sigma_k(D)$ does not depend on choices of the expression  $\pi = U_{i_1}\dots U_{i_n }$. Let $\mathcal{M}_k(D)$ be the monoid generated by elementary transpositions of $\mathfrak{S}_k(D)$. We have constructed the map 
\[\Lambda_D(\pi) = \sigma_{k-1} \circ \Gamma_D : \mathbb{Q}[\mathcal{M}_k(D)] \rightarrow \mathcal{B}
\] from the monoid ring of $\mathcal{M}_k(D)$. 
To show it defines a map from $\mathbb{Q}[\mathfrak{S}_k(D)]$, 
it is enough to show the kernel of $\Lambda_D$ includes the kernel $\mathcal{K}$ of the map $\mathbb{Q}[\mathcal{M}_k(D) ]\rightarrow \mathbb{ Q}[\mathfrak{S}_k(D)]$. The kernel $\mathcal{K}$ is the both-sided ideal generated by 
\begin{itemize}
\item [(G1)] $U_i^2 - id$,
\item [(G2)] $U_i U_j - U_j U_i = id\ (|i-j| >1)$,
\item [(G3)] $U_i U_{i-1} U_i - U_{i-1} U_i U_{i-1}$.
\end{itemize}
Here the $U_i$ s in each relation must be elements of the same $\mathfrak{S}_{i_l}$ of $\mathfrak{S}_k(D) = \prod_{l=1}^{s} \mathfrak{S}_{i_l}$.

Let $D$ be a colored graph of the even case

\begin{figure}[htpb]
\centering
\begin{tikzpicture}[x=0.75pt,y=0.75pt,yscale=0.4,xscale=0.4, baseline=0pt, line width = 1pt] 
\draw (-10,0)--(-250,0); 
\draw (10,0)--(200,0); 

\draw [dash pattern={on 5pt off 4pt}]  (-200,0)--(-200,100); 
\draw [dash pattern={on 5pt off 4pt}]  (-150,0)--(-150,100); 
\draw [dash pattern={on 5pt off 4pt}]  (-100,0)--(-100,100); 
\draw [dash pattern={on 5pt off 4pt}]  (150,0)--(150,100); 
\draw [dash pattern={on 5pt off 4pt}]  (100,0)--(100,100); 

\draw (-10,20)..controls (20,0) and (-20,-10)..(0,-30);
\draw (0,20)..controls (30,0) and (-10,-10)..(10,-30);

\draw (-175,-30) node {$(k)$};
\draw (-125,-30) node {$(l)$};
\draw (125,-30) node {$(m)$};

\draw [fill = black] (-200, 0) circle (7);
\draw [fill = black] (-150, 0) circle (7);
\draw [fill = black] (-100, 0) circle (7);
\draw [fill = black] (100, 0) circle (7);
\draw [fill = black] (150, 0) circle (7);
\end{tikzpicture},
\end{figure}
\noindent where the $(i-1)$th, $i$-th, $(i+1)$-th, $j$-th, $(j+1)$-th dashed edges are drawn. Without loss of genericity, we assume that $k<l<m$. We show that $\Lambda_D$ vanishes on (G1), (G2) and (G3). A completely similar proof works for the odd case.

\noindent (G1)
\begin{align*}
&\Lambda_D(U_i^2) \\
= &  \sigma_{k-1}(S_i D + S_i U_i D) \\
= &  \sigma_{k-1}(S_i D - S_i D) \\
= & 0. \\
\end{align*}

\noindent (G2)
\begin{align*}
&\Lambda_D(U_i U_j - U_j U_i) \\
 = & \sigma_{k-1}(S_j D + S_i U_j D - S_i D - S_j U_i D) \\
 = & \sigma_{k-1}(S_j D - S_j U_i D-S_i D+ S_i U_j D ). \\
\end{align*}
By STU relation in $\mathcal{I}_{k-1}$, 
\[
S_j D - S_j U_i D = D_s= S_i D- S_i U_j D
\]
for the graph $D_s$ in Figure \ref{property2}, which has $k-2$ external vertices. 
Since $ \sigma_{k-1}$ satisfies ($\Sigma 4$) by the induction hypothesis, we have $\Lambda_D(U_i U_j - U_j U_i) = 0$.

\begin{figure}[htpb]
\begin{center}
\includegraphics[width = 10cm]{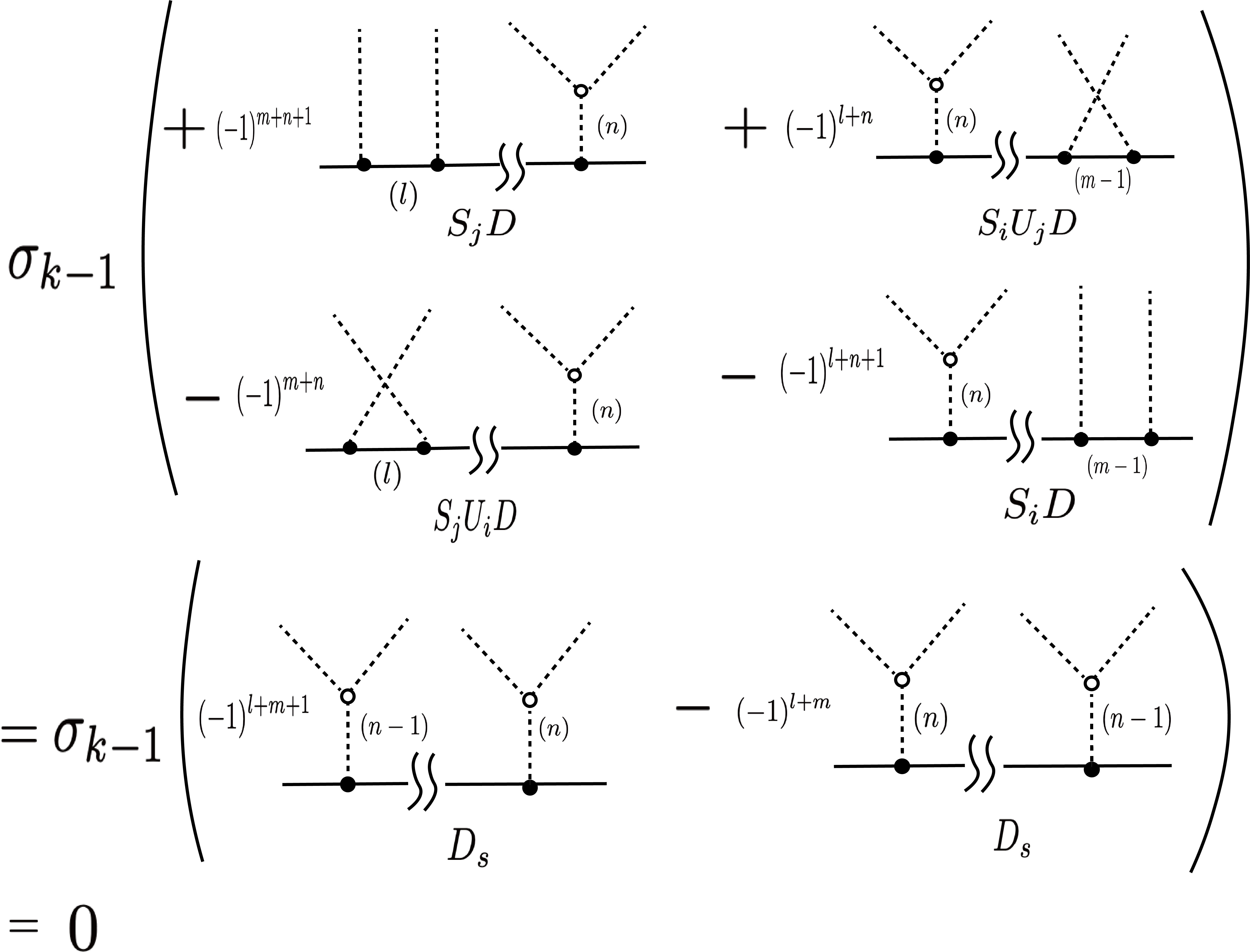}
\caption{Graph $D_s$ in (G2) (even case) }
\label{property2}
\end{center}
\end{figure}

\noindent (G3) 
\begin{align*}
&\Lambda_D(U_i U_{i-1} U_i - U_{i-1} U_i U_{i-1})\\
= & \sigma_{k-1}(S_i D + S_{i-1} U_i D + S_i U_{i-1} U_i D- S_{i-1} D- S_i U_{i-1} D - S_{i-1} U_i U_{i-1} D) \\
= & \sigma_{k-1}(S_i D - S_{i-1} U_i U_{i-1} D + S_{i-1} U_i D - S_i U_{i-1} D  + S_i U_{i-1} U_i D - S_{i-1} D) \\
= & \sigma_{k-1} (D_I + D_H + D_X). \\
\end{align*}
Where $D_I$, $D_H$ and $D_X$ are graphs in Figure \ref{property3} and $D_I + D_H + D_X$  satisfies an IHX relation. We used STU relation in $\mathcal{I}_{k-1}$ in the last equation, 
Since $ \sigma_{k-1}$ satisfies ($\Sigma 3$) by the induction hypothesis, we have $ \Lambda_D(U_i U_{i-1} U_i - U_{i-1} U_i U_{i-1}) = 0$.

\begin{figure}[htpb]
\begin{center}
\includegraphics[width = 10cm]{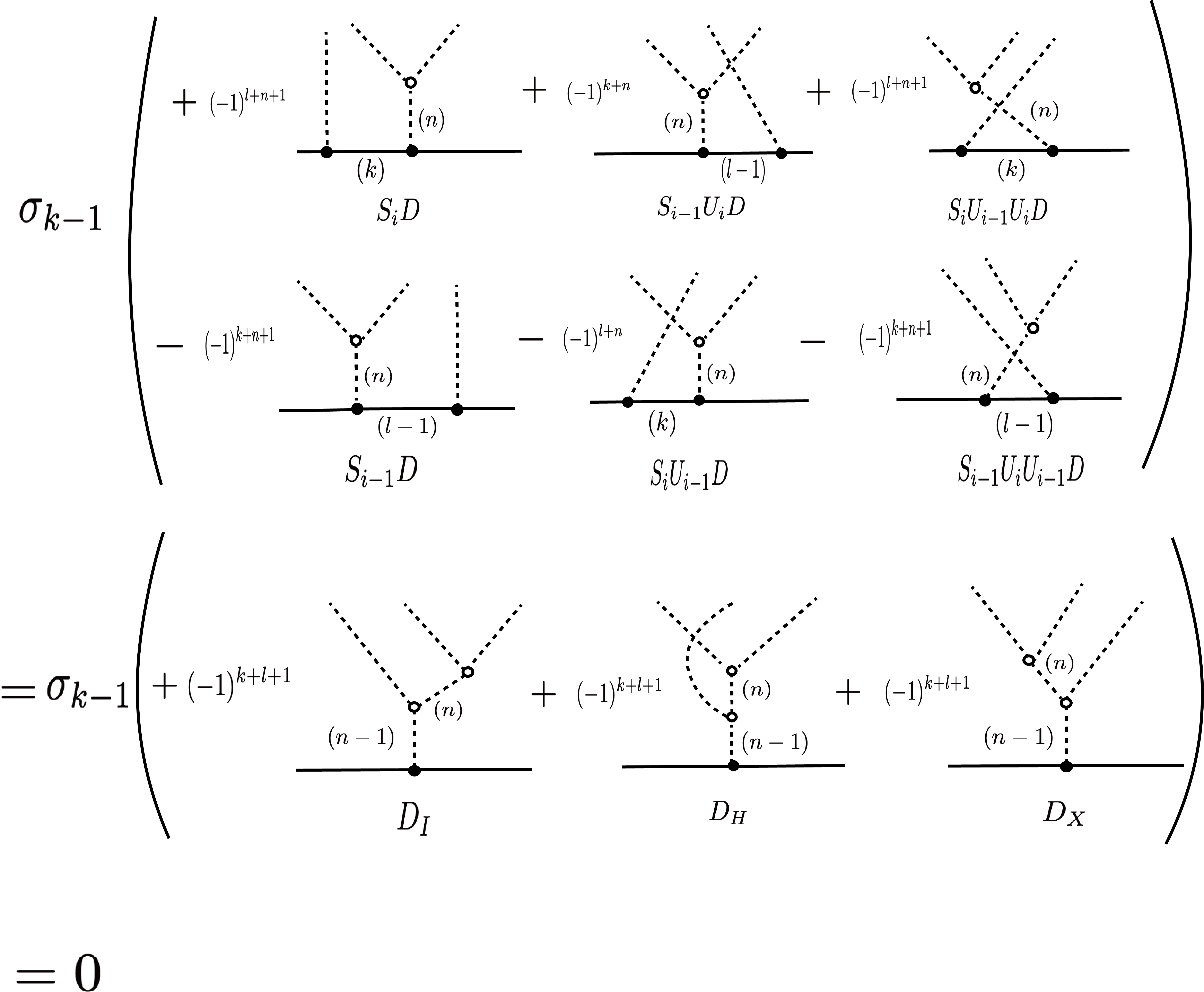}
\caption{Graphs $D_I, D_H, D_X$ in (G3)  (even case)}
\label{property3}
\end{center}
\end{figure}

Next we prove that $\sigma_k$ satisfies ($\Sigma1$) ($\Sigma2$) ($\Sigma3$)  ($\Sigma4$).

\noindent ($\Sigma1$) 
For a hairy graph $C$, 
\begin{align*}
\sigma_k \circ \chi_k (C) = \sigma_k (C) = C.
\end{align*}

\noindent ($\Sigma2$) 
If the orientation of $D$ is reversed (resp. preserved), the orientations of $U_i D$ and $S_i D$ are also reversed (resp. preserved). Hence, the orientation of $\Gamma_D (\pi)$ is reversed (resp. preserved). By the induction hypothesis, $\sigma_{k-1} $ vanishes on orientation relations. Since $\sigma_k$ is the sum of $\Lambda_D(\pi) = \sigma_{k-1} \Gamma_D( \pi) $ for $\pi \in \mathfrak{S}_k(d)$, $\sigma_k$(orientation relations)$= 0$.
  
\noindent ($\Sigma3$) 
Let $D_I + D_H + D_X $ be an IHX relation. Then $\Gamma_{D_I}(\pi) + \Gamma_{D_H} (\pi)+ \Gamma_{D_X}(\pi)$ also satisfies an IHX relation, because the procedure does not change the part IHX occurs.  Since $ \sigma_{k-1}$ satisfies ($\Sigma 3$) by the induction hypothesis,
\begin{align*}
&\Lambda_{D_I}(\pi)+ \Lambda_{D_H}(\pi) + \Lambda_{D_X}(\pi)\\
 = &\sigma_{k-1}(\Gamma_{D_I} (\pi)+ \Gamma_{D_H} (\pi) + \Gamma_{D_X}(\pi))\\
 = &0.
 \end{align*}
 
\noindent ($\Sigma4$) 
Let  $S_i D= D - U_i D $ be an STU relation. Then we have
\begin{align*}
& \sigma_k(D - U_i D ) \\
= & \frac{1}{|\mathfrak{S}_k(D)|} \sum_{\pi \in \mathfrak{S}_k(D)} (\Lambda_D(\pi) -  \Lambda_{U_iD} (\pi)) \\
= & \frac{1}{|\mathfrak{S}_k(D)|} \sum_{\pi \in \mathfrak{S}_k(D)} (\Lambda_D(\pi) -  \left(\Lambda_{D} (\pi U_i) - \Lambda_{D} (U_i)\right))\\
= & \Lambda_{D} (U_i) =  \sigma_{k-1} \Gamma_{D} (U_i) = \sigma_{k-1} (S_i D) = \sigma_k (S_i D). \\
\end{align*}

\end{proof}

 \subsection{$\iota: \mathcal{A}^c _{BCR} \rightarrow \mathcal{A}_{BCR}$ is an isomorphism}

\begin{proof}[Proof of Theorem \ref{Main Result 2}]
We construct the inverse $\kappa_{\ast}$ (resp. $\overline{\kappa}_{\ast}$) of the map $\iota_{\ast}$ (resp. $\overline{\iota}_{\ast}$).
Namely, we construct the map $\kappa: \mathcal{D} \rightarrow \mathcal{A}^c_{BCR}$  (resp. $\kappa: \mathcal{D} \rightarrow \overline{\mathcal{A}}^c_{BCR}$) which vanishes on STU, IHX and orientation relations (and chord relations). Note that by Lemma \ref{IHX and STU} below, an IHX relation can be written a sum of several STU relations. 
Let $D$ be a (colored) BCR graph. By using STU relations repeatedly, we can transform the graph $D$ to a linear combination $H =\sum_k a_k D_k$, where each $D_k$ is a BCR chord diagram. We call this operation resolution of $D$. We set $\kappa(D)$ as the equivalence class of $H$ in $\mathcal{A}^c_{BCR}$. By construction, the map $\kappa$, if it is well-defined, induces $\kappa_{\ast}: \mathcal{A}_{BCR} \rightarrow \mathcal{A}^c_{BCR}$. We can easily check $\kappa_{\ast} \iota_{\ast} = id$ and $\iota_{\ast} \kappa_{\ast}= id$. 

However, we must show that the resulting $H = H(D)$ does not depend, up to $4T$ relations, on choices of procedures of STU relations. We show this by induction on the number of internal vertices $\#W(D)$. 

The case  $\#W(D) = 0$ is trivial because no STU relation is performed. Let $\#W(D) = 1$, and let $w$ be the unique internal vertex. Since $w$ has three adjacent dashed edges $e_1$, $e_2$ and $e_3$, there are three choices of STU relations. Each choice produces a sum of two (or one) chord diagrams. We can easily see that the differences between these combinations are exactly $4T$ relations.

Next, we assume that the case $\#W(D) \leq i$ is proved and show the case $\#W(D) = i+1$ $(i\geq1)$.
Let an edge $e_1$ of the graph $D$ connect an external vertex $v_1$ and an internal vertex $w_1$. Let another edge $e_2$ of $D$ connect an external vertex $v_2$ and an internal vertex $w_2$. 
First, we assume $w_1 \neq w_2$.
Consider the two procedures to perform STU relations. One performs the first STU relation to $e_1$ and gives a sum of graphs $D_{1t} +D_{1u}$, while another performs the first STU relation to $e_2$ and gives $D_{2t} + D_{2u}$. Note that after the first choice of STU relation, the resulting combination of chord diagrams does not depend, up to $4T$ relations, on the remaining choices of STU relations. Then we can assume that the second STU relation of $D_{1t} +D_{1u}$ is performed to the edge $e_2$ (we write this procedure as $e_1 \rightarrow e_2$), while the second STU relation of $D_{2t} + D_{2u}$ is performed to the edge $e_1$ (we write this as $e_2 \rightarrow e_1$). But the two procedures yield the same combination of graphs.

So the remaining case is the case $w = w_1 = w_2$ and $w$ is connected to another internal vertex. If there is an edge $e_3$ that connects external vertex $v_3$ and an internal vertex $w_3 \neq w$, the procedure $(e_1 \rightarrow e_3)$ is equal to $(e_3 \rightarrow e_1) = (e_3 \rightarrow e_2) = (e_2 \rightarrow e_3)$. So let $w$ be connected to a block $B$ which consists of only internal vertices (and dashed edges). We see in Lemma \ref{Slidingargument} that in both cases ($e_1$ first and  $e_2$ first), the result of STU relations to $B$ after the STU relation to $e_i$ vanishes by $4T$ relation. 

This ends the proof of well-definedness of $\kappa$. Finally, we show that the map $\kappa$ sends chord relations to chord relations. Let $\sum_i D_i (= 0)$ be a chord relation such that $D_1$ has an internal vertex $w_1$ connected by an edge $e_1$ to an external vertex $v_1$. Consider performing the STU relation to $e_1$ and transform $D_1 = D_{1s}$ to $D_{1t} + D_{1u}$. Since graphs $D_i$ are related by chord relations, there are corresponding edges $e_i = (v_i, w_i)$ and corresponding STU relations $D_i = D_{is} = D_{it} + D_{iu}$ for each $D_i$. We can observe that $\sum_i D_{it}$ and $\sum_i D_{iu}$ are chord relations.

\end{proof}

\subsection{IHX in terms of STU}

\begin{lemma}
\label{IHX and STU}
Recall that  $\mathcal{D}$ is the space generated by BCR graphs with at least one external vertex. 
Then an IHX relation in $\mathcal{D}$ can be written as a sum of STU relations. 
\end{lemma}

\begin{proof}
Let a graph $D$ have an ``$I$'' part, namely two internal vertices $w_1$ and $w_2$ connected by a dashed edge $e$. Consider the IHX relation performed on the edge $e$. 
We have three graphs $D = D_1$, $D_2$ and $D_3$. We want to show $D_1 + D_2 + D_3 = 0$ by only using STU relations. 

Using STU relation simultaneously to  $D_1$, $ D_2$ and  $D_3$, we can assume that $w_1$ is connected to an external vertex $v$ by some dashed edge $f$. 
Then perform the STU relation to $f$ and $e$ in this order. 
This yields four or two graphs for each $D_i$, depending on the number of solid edges the external vertex $v$  has. We can easily see that these twelve (or six) graphs are canceled. See Figure \ref{IHXSTU}.

\begin{figure}[htpb]
\begin{center}
\includegraphics[width = 9cm]{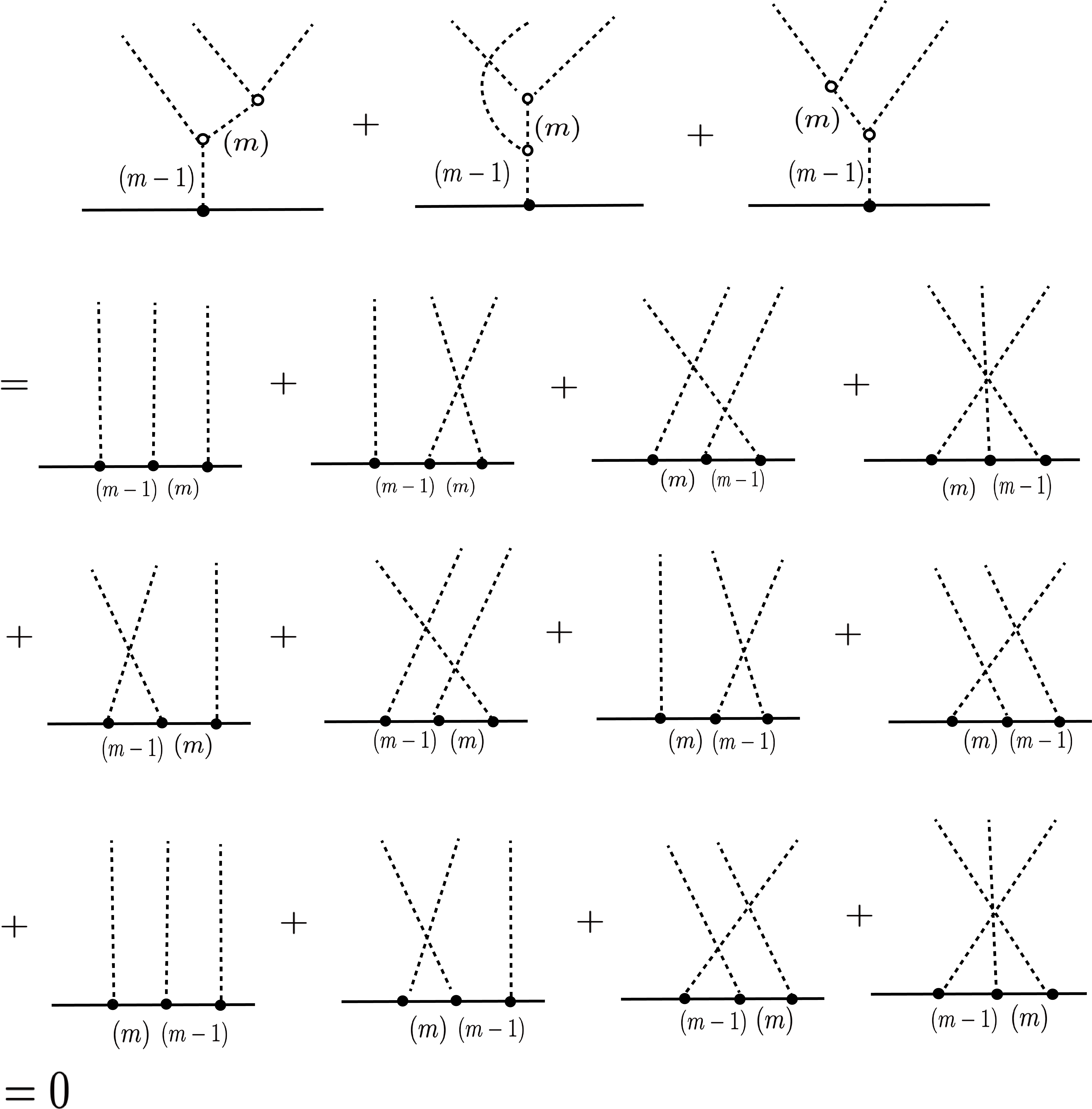}
\caption{IHX is a sum of STU relations}
\label{IHXSTU}
\end{center}
\end{figure}

\end{proof}

\subsection{Sliding of chords}

\begin{lemma}[Sliding of chords]
\label{Slidingargument}
Let $D$ be a BCR chord diagram. 
Let $w_1$ be an external vertex and let $w_1$ have one dashed edge $(w_0, w_1)$ that is connected to an external vertex $w_0$. 
Let $v_1, v_2, \dots, v_{2n}$ be other external vertices. Assume that these vertices satisfy the following. See Figure \ref{slidingargument}.
\begin{itemize}
\item $w_1$ is connected by one solid edge $e$ to $v_1$
\item $v_1, v_2, \dots, v_{2n}$ are connected by $n$ chords among them.
\item $v_i$ and $v_{i+1}$ are connected by one solid edge.
\item $v_i$ has no other solid edge except for at most one solid edge connected to $v_{2n}$. We write this exceptional edge as $f$.
\end{itemize}

\begin{figure}[htpb]
\begin{center}
\includegraphics[width = 9cm]{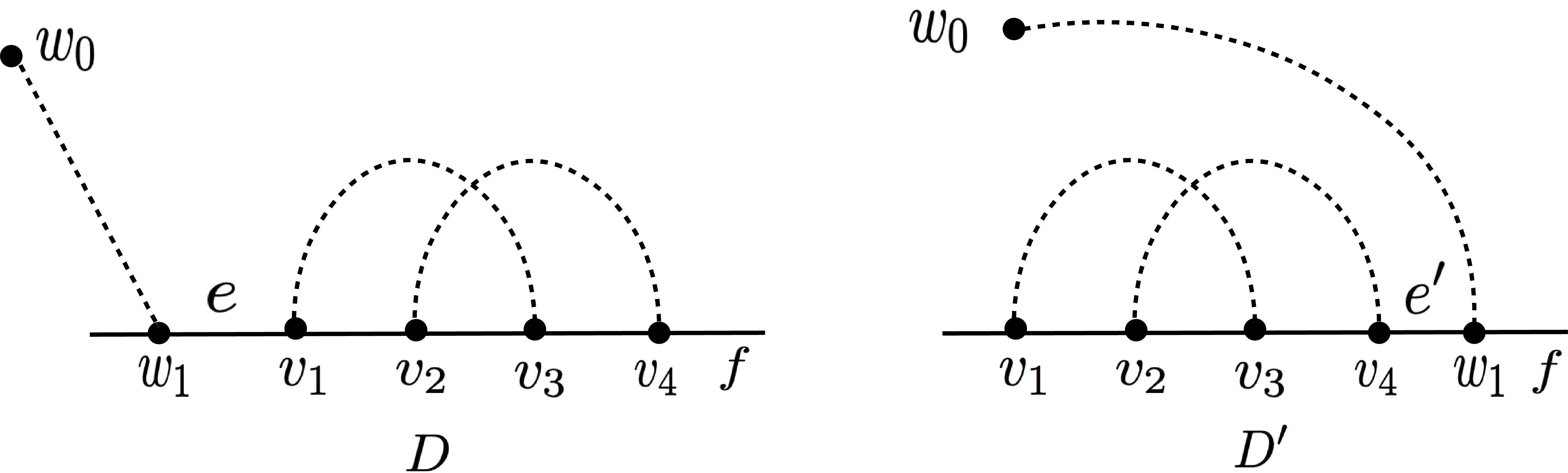}
\caption{Example of sliding argument}
\label{slidingargument}
\end{center}
\end{figure}

We construct another chord diagram $D^{\prime}$ from $D$ by remove the solid edge $e = (v_1, w_1)$ and inserting a new solid edge $e^{\prime} = (v_{2n}, w_1)$ to $f$. 

Then by using $4T$ relation,  $D$ and $D^{\prime}$ are equivalent. 
\end{lemma}

\begin{proof}
4T relation can be written as $D_1 + D_2 + D_3 + D_4 = 0$, where $D_i (i = 1,2,3,4)$ are four graphs obtained by four ways to connect an end $w_1$ of dashed edge to the sides of some fixed chord. 
Consider the sum of $4n$ graphs obtained by all the ways to connect $w_1$ to the sides of the $n$ chords. By $4T$ relation, this sum is equal to $0$. On the other hand, observe that in the sum there are $2n-1$ pairs that consist of the two isomorphic graphs with the opposite sign. (See Figure \ref{slidingargument2}.)The remaining graphs are $D$ and $D^{\prime}$.

\begin{figure}[htpb]
\begin{center}
\includegraphics[width = 9cm]{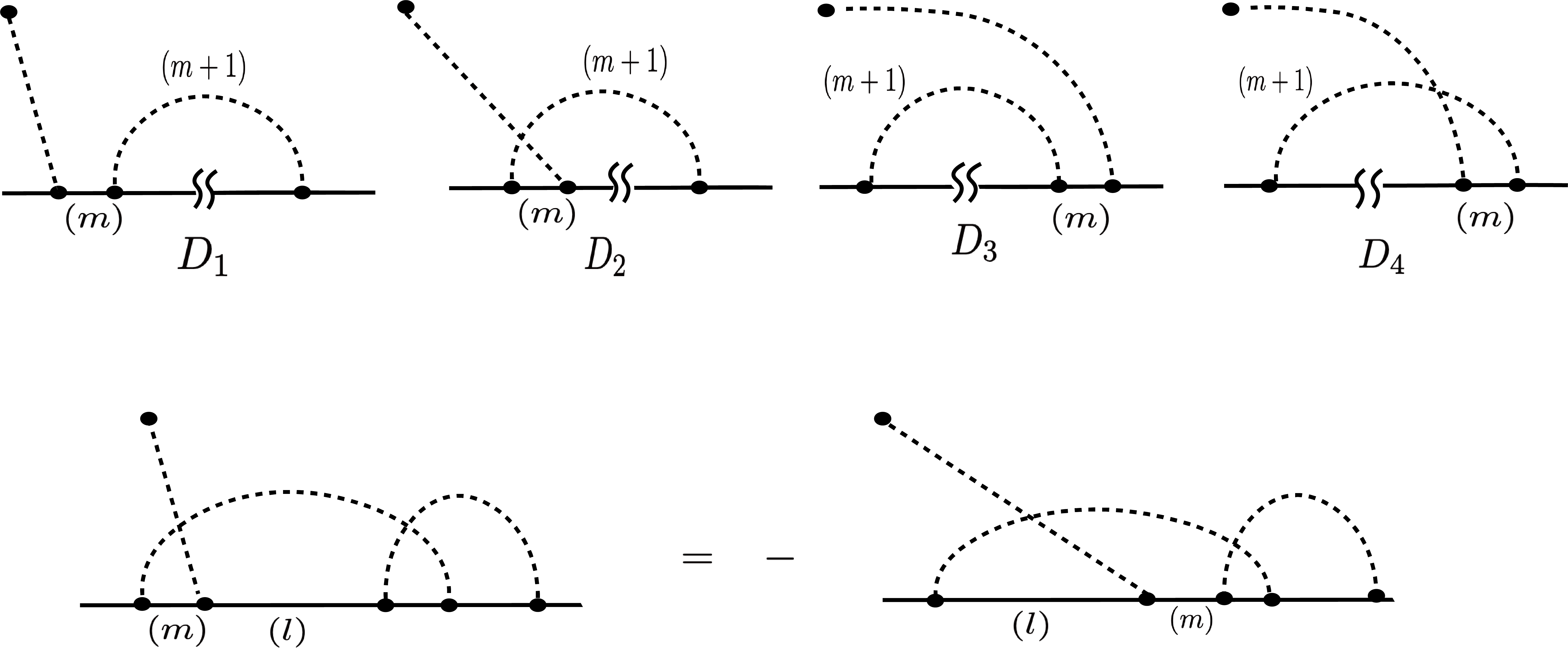}
\caption{Four graphs $D_1$, $D_2$,  $D_3$, $D_4$ (above). A pair of isomorphic graphs with the opposite sign (below).}
\label{slidingargument2}
\end{center}
\end{figure}

\end{proof}

\end{document}